%% file: elsarticle-template-III.tex
\documentclass[preprint,12pt,3p]{elsarticle}
\usepackage{amssymb, amsmath}
\usepackage{amsthm}
\usepackage{color}
\usepackage{enumitem}

\newtheorem{theorem}{Theorem}[section]
\newtheorem{lemma}[theorem]{Lemma}

\theoremstyle{definition}
\newtheorem{definition}[theorem]{Definition}

\newtheorem{example}[theorem]{Example}

\theoremstyle{remark}

\begin{document}

\begin{frontmatter}

\title{Forced oscillations of a massive point on a compact surface with boundary}

\author{Ivan Polekhin}
%\address[label1]{Address One}
%\address[label2]{Address Two\fnref{label4}}

%\cortext[cor1]{I am corresponding author}
%\fntext[label3]{I also want to inform about\ldots}
%\fntext[label4]{Small city}

\ead{ivanpolekhin@gmail.com}
%\ead[url]{author-one-homepage.com}

%\author[label5]{Author Two}
%\address[label5]{Some University}
%\ead{author.two@mail.com}

%\author[label1,label5]{Author Three}
%\ead{author.three@mail.com}

\begin{abstract}
We present sufficient conditions for the existence of a periodic solution for a class of systems describing the periodically forced motion of a massive point on a compact surface with a boundary. 
\end{abstract}


\end{frontmatter}

% \linenumbers

\section{Brief introduction}
\label{sec1}

In 1922, G.\,Hamel proved~\cite{hamel1922erzwungene} that equations describing motion of a periodically forced pendulum have at least one periodic solution. Since then, many results concerning periodic solutions in pendulum-like systems have been obtained by various authors including results for a one-dimensional forced pendulum~\cite{mawhin2000global}, a result by M.\,Furi and M.\,P.\,Pera~\cite{furi1991forced} who showed that a frictionless spherical pendulum also have forced oscillations and a work by V.\,Benci and M.\,Degiovanni~\cite{benci1990periodic} in which motion of a massive point on a compact boundaryless surface with friction is studied and sufficient conditions for existence of forced oscillations are presented. As far as we know, the case of compact surface with boundary is far less developed.

However, surfaces with boundaries naturally appear in various mechanical systems. For instance, in a book~\cite{courant1941elementary} by R.\,Courant and H.\,Robbins the authors consider the problem which states that for an inverted planar pendulum placed on a floor of a train carriage, for any law of motion for the train, there always exists at least one initial position such that the pendulum, starting its motion from this position with zero generalized velocity, moves without falling for an arbitrary long time. Here, the compact surface is a half-circle and its boundary is the two-pointed set.

Topological ideas, which lie in the basis of the above result, can be rigorously justified~\cite{polekhin2014examples} --- in~\cite{courant1941elementary} some details are omitted --- and generalized for different types of systems. Moreover, it was proved~\cite{polekhin2014examples} that for an inverted pendulum with a periodic law of motion for its pivot point, there exists a periodic solution along which the pendulum never becomes horizontal, i.e. it never falls. This was obtained as an application of a topological theorem by R.\,Srzednicki, K.\, W{\'o}jcik, and P.\,Zgliczy{\'n}ski~\cite{srzednicki2005fixed}.

In the current paper we further develop this result~\cite{polekhin2014examples} and present sufficient conditions for the existence of a periodic solution for a class of systems describing periodically forced motion with friction of a massive point on a compact surface with a boundary and non-zero Euler-Poincar\'{e} characteristic. We prove that if for the considered system all solutions that are tangent to the boundary are externally tangent to it, then there exists at least one periodic solution that never reaches the boundary.

\section{Main result}
\label{sec2}

\subsection{Governing equations}
\label{subsec2_1}
In this subsection, we introduce governing equations for a mechanical system consisting of a massive point moving with friction-like interaction on a surface and prove a lemma, which we are going to use further in our main theorem. For the sake of simplicity, we assume that all manifolds and considered functions are smooth (i.e. $C^\infty$). 

Let $M$ be a compact connected two-dimensional manifold with a boundary embedded in~$\mathbb{R}^3$. Manifold $M$ describes the surface on which a massive point moves. Its boundary is a finite collection of curves, which are homeomorphic to circles. We also assume that the point moves with friction, which we will specify further below.

In our further consideration, we will study the behaviour of our system in a vicinity of $\partial M$. In this regard, it is convenient to consider an enlarged manifold $M^+$. Let $M^+$ be a boundaryless connected two-dimensional manifold also embedded in $\mathbb{R}^3$ such that $M \subset M^+$. Therefore, the motion of the massive point can be described by a function of time $q \colon \mathbb{R} \to M^+$. Note that there are infinitely many possibilities for constructing $M^+$ but for our use they are all the same.

In general form, the equations of motion can be written as follows
\begin{equation*}
m\ddot q = F + F_{friction} + F_{constraint}.
\end{equation*}
Here $m$ is the mass of the point; $F \colon \mathbb{R} / T\mathbb{Z} \times TM^+ \to \mathbb{R}^3$ is a $T$-periodic force acting on the point;  $F_{friction} \colon \mathbb{R} / T\mathbb{Z} \times TM^+ \to \mathbb{R}^3$ is a friction-like force which, for a given $t$, $q$ and $\dot q$, we assume to have the following form 
\begin{equation*}
\label{eqFfric1}
F_{friction} = -\dot q \gamma(t,q,\dot q), \quad  \gamma\colon \mathbb{R} / T\mathbb{Z} \times TM^+ \to \mathbb{R}.
\end{equation*}
The force of constraint has usual form $F_{constraint} = \lambda n_q$, where $\lambda \in \mathbb{R}$ and $n_q$ is a normal vector to $M^+$ at point $q$.

Finally, assuming without loss of generality that $m = 1$, we obtain the following equations of motion
\begin{equation}
\label{MainEq}
\begin{aligned}
&\dot q = p,\\
&\dot p = F(t,q,p) - p \gamma(t,q,p) +  \lambda n_q.
\end{aligned}
\end{equation}
Note that one can get rid of unknown parameter $\lambda$ in~(\ref{MainEq}) in the usual way by projecting the right-hand sides of the above equations to $T_qM^+$.
\begin{lemma}
\label{Lemma-I}
Suppose that there exists a constant $d > 0$ such that in~(\ref{MainEq})
\begin{equation}
\label{GammaEq}
\inf_{\substack{t \in [0, T],\, q \in M \\ \| p \| > d}} \gamma(t,q,p) > 0,
\end{equation}
and $F$ is a bounded function, then for some $c > 0$ along the solutions of~(\ref{MainEq})
\begin{equation*}
\frac{d}{dt} T \Big|_{T = c} < 0, \quad T = \frac{p^2}{2}.
\end{equation*}
\end{lemma}
\begin{proof}
By direct calculation from~(\ref{MainEq}), we have
\begin{equation*}
\frac{d}{dt} T = p \cdot \dot p = (p, F) - p^2\cdot\gamma(t,q,p) + \lambda(p,n_q).
\end{equation*}
Since $n_q$ is a normal vector to $M^+$, then $(p, n_q) = 0$. Moreover, from~(\ref{GammaEq}) we have that if $\|p\|$ is large enough, then $ \|p\|^2\cdot |\gamma(t,q,p)| > \| (p, F) \|$.
\end{proof}

\subsection{Auxiliary constructions and results}
\label{subsec2_2}

The approach developed in~\cite{srzednicki2005fixed} is based on the ideas of the Wa\.{z}ewski method and the Lefschetz-Hopf theorem. In this subsection we introduce some definitions and a result from~\cite{srzednicki2005fixed} which we slightly modify for our use.
\par
Let $v \colon \mathbb{R}\times M \to TM$ be a time-dependent vector field on a manifold $M$
\begin{equation}
\label{eq1}
\dot x = v(t, x).
\end{equation}
For $t_0 \in \mathbb{R}$ and $x_0 \in M$, the map $t \mapsto x(t,t_0,x_0)$ is the solution for the initial value problem for the system~(\ref{eq1}), such that $x(0,t_0,x_0)=x_0$. If $W \subset \mathbb{R}\times M$, $t\in\mathbb{R}$, then we denote
\begin{equation*}
W_t=\{x \in M \colon (t,x) \in W\}.
\end{equation*}
\begin{definition}
Let $W \subset \mathbb{R} \times M$. Define the {exit set} $W^-$ as follows. A point $(t_0,x_0)$ is in $W^-$ if there exists $\delta>0$ such that $(t+t_0, x(t,t_0,x_0)) \notin W$ for all $t \in (0,\delta)$.
\end{definition}
\begin{definition}
We call $W \subset \mathbb{R}\times M$ a {Wa\.{z}ewski block} for the system~(\ref{eq1}) if $W$ and $W^-$ are compact.
\end{definition}
\begin{definition}
A set $W \subset [a,b] \times M$ is called a simple periodic segment over $[a,b]$ if it is a Wa\.{z}ewski block with respect to the system~(\ref{eq1}), $W = [a, b] \times Z$, where $Z \subset M$, and $W^-_{t_1} = W^-_{t_2}$ for any $t_1, t_2 \in [a, b)$.
\end{definition}
\begin{definition}
Let $W$ be a simple periodic segment over $[a,b]$. The set $W^{--} = [a,b] \times W_a^-$ is called the essential exit set for $W$.
\end{definition}
In our case, the result from~\cite{srzednicki2005fixed} can be presented as follows.

\begin{theorem} 
\label{th1}
\cite{srzednicki2005fixed} Let W be a simple periodic segment over $[a,b]$. Then the set
\begin{equation*}
U = \{ x_0 \in W_a \colon x(t-a,a,x_0) \in W_t\setminus W_t^{--}\,\mbox{for all}\,\, t \in [a,b] \}
\end{equation*}
is open in $W_a$ and the set of fixed points of the restriction $x(b-a,a,\cdot)|_U \colon U \to W_a$ is compact. Moreover, the fixed point index of $x(b-a,a,\cdot)|_U$ can be calculated by means of the Euler-Poincar\'{e} characteristic of $W$ and $W^-_a$ as follows
\begin{equation*}
\mathrm{ind}(x(b-a,a,\cdot)|_U) = \chi(W_a) - \chi(W^-_a).
\end{equation*}
In particular, if $\chi(W_a) - \chi(W^-_a) \ne 0$ then $x(b-a,a,\cdot)|_U$ has a fixed point in $W_a$.
\end{theorem}

\subsection{Main theorem}
\label{subsec2_3}

In this subsection, we prove our main result and illustrate it with examples closely related to the problem from~\cite{courant1941elementary} concerning the falling-free motion of an inverted pendulum with a moving pivot point.

\begin{theorem}
\label{MyMainTh}
Suppose that for~(\ref{MainEq}) the following conditions are satisfied 
\begin{enumerate}
\item The Euler-Poincar\'{e} characteristic of $M$ is non-zero,
\item There exists a constant $d > 0$ such that
\begin{equation*}
\inf_{\substack{t \in [0, T],\, q \in M \\ \| p \| > d}} \gamma(t,q,p) > 0,
\end{equation*}
\item $F$ is a bounded function,
\item For any $t_0 \in \mathbb{R}$ and $(q_0, p_0) \in T(\partial M)$ there is an $\varepsilon > 0$ such that
\begin{equation}
\label{Cond-IV}
q(t,t_0,q_0,p_0) \notin M, \quad \mbox{for all} \quad t \in (0, \varepsilon).
\end{equation}
\end{enumerate}
Then there exists a solution $(q, p) \colon \mathbb{R} \to TM^+$ of~(\ref{MainEq}) such that 
\begin{equation*}
q(t) = q(t+T), \quad p(t) = p(t+T), \quad q(t) \in M \setminus \partial M, \quad\mbox{for all}\quad t \in \mathbb{R}.
\end{equation*}
\end{theorem}
\begin{proof}
Let us consider the following compact subset $W$ of~$[0,T]\times TM^+$
\begin{equation*}
W = \{ 0 \leqslant t \leqslant T, (q, p) \in TM^+ \colon q \in M, \frac{p^2}{2} \leqslant c \}
\end{equation*}
where $c>0$ is the constant obtained from~(\ref{Lemma-I}).

From~(\ref{Lemma-I}) we also have that if $(t,q,p) \in W^{--}$ then $q \in \partial M$. Let $\nu_q \in T_q M^+$ be a normal vector to $\partial M$ at point $q \in \partial M$ such that if for $p \in T_q M^+$ we have $(\nu_q, p) > 0$, then a solution starting from $(t,q,p)$ at least locally leaves $M$. From the above definition of $\nu_q$, we obtain
\begin{equation*}
\{  0 \leqslant t \leqslant T, (q, p) \in TM^+ \colon q \in M, \frac{p^2}{2} \leqslant c, (\nu_q, p) > 0 \} \subset W^{--}.
\end{equation*}
Moreover, we also have
\begin{equation*}
\{  0 \leqslant t \leqslant T, (q, p) \in TM^+ \colon q \in M, \frac{p^2}{2} \leqslant c, (\nu_q, p) < 0 \} \cap W^{--} = \varnothing.
\end{equation*}
Since we assume~(\ref{Cond-IV}), then $W^{--}$ is compact:
\begin{equation*}
W^{--} = \{  0 \leqslant t \leqslant T, (q, p) \in TM^+ \colon q \in M, \frac{p^2}{2} \leqslant c, (\nu_q, p) \geqslant 0 \}.
\end{equation*}
Therefore, $W^{-}$ is also compact and $W$ is a simple periodic segment over $[0,T]$.

  Since $M$ is compact, then boundary $\partial M$ consists of a finite number of curves that are homeomorphic to circles. Moreover, $W_0^- = W_0^{--}$ is homotopic to a finite number of circles and we have $\chi(W_0^-) = 0$. Finally, since $\chi(W_0) = \chi(M) \neq 0$, then we can apply~(\ref{th1}).
\end{proof}

The following example presents the quite counter-intuitive behaviour of a mechanical system in which a periodically forced inverted pendulum is moving periodically and never falls, i.e. it never becomes horizontal. Assuming motion with friction, this means that in the considered system, we observe both effects described in~\cite{furi1991forced} and~\cite{courant1941elementary} simultaneously.

\begin{example}
Consider an inverted spherical pendulum in a gravitational field moving with viscous friction $F_{friction} = -\gamma p$, $\gamma>0$. Suppose that its pivot point moves with a prescribed periodic law parallel to the horizontal plane. Then there exists a periodic solution such that along this solution the pendulum never becomes horizontal, i.e. it never falls.
In this case, $M$ is a half-sphere and $\partial M$ is the horizontal great circle. Condition~(\ref{Cond-IV}) is satisfied since when the pendulum is horizontal, there is no vertical force of inertia and the only vertical force that acts on the massive point is the force of gravity.
\end{example}

The above example can be generalized to the case of a compact surface if we assume that the considered surface is vertical at its boundary and the boundary itself is horizontal.

\begin{figure}[h!]
    \centering
    \def\svgwidth{240 pt}
    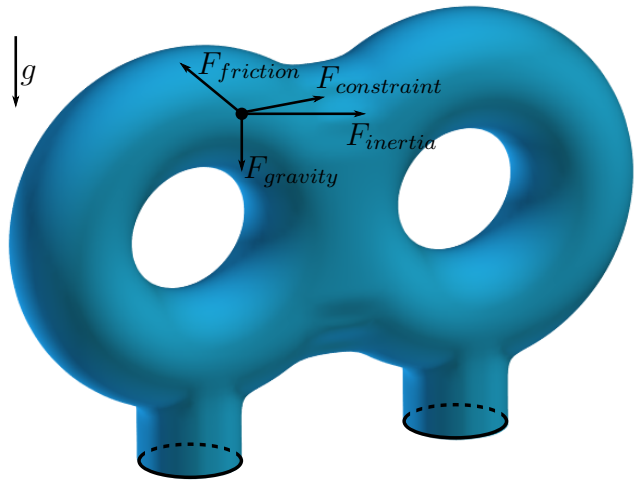
    \caption{An example of a surface with non-zero Euler-Poincar\'{e} characteristic, vertical at its horizontal boundary and which is above its boundary (at least locally). If we assume that the surface is moving parallel to the horizontal plane, then $F_{inertia}$ is horizontal.}
    \label{Fig1}
\end{figure}

\begin{example}
Consider a massive point on a moving compact surface $M$ with a boundary. Here we also assume that the point is in a gravitational field, it moves with viscous friction, the surface moves with a prescribed periodic law parallel to the horizontal plane and the Euler-Poincar\'{e} characteristic of the surface is non-zero. To satisfy~(\ref{Cond-IV}), we also assume that all boundary curves are horizontal and the surface is vertical at the boundary, i.e. its tangent planes are vertical. Moreover, just like in the case of an inverted pendulum, we suppose that locally our surface is above its boundary curves. The last condition means that any solution starting at $\partial M$ and tangent to it at the initial moment of time `falls down through the boundary', i.e. locally it leaves $M$. Indeed, the force of inertia is always horizontal and if the point is at the boundary, then the force of constraint is also horizontal. Therefore, the only non-zero vertical force is the force of gravity. From~(\ref{MyMainTh}), it also follows that in this case, which can be considered as a generalized inverted pendulum with a moving pivot point, there exists a periodic solution that always stays in $M \setminus \partial M$.
\end{example}

\section{Conclusion}
\label{sec3}

%%\appendix

%%\section{Section in Appendix}
%%\label{appendix-sec1}

In conclusion, the obtained periodicity seems to be a property of the considered type of systems in the frictionless case as well. Naturally, one of the main arguments toward this is that the constant of friction can be chosen to be arbitrarily small. The existence of a periodic solution without falling for a planar pendulum moving without friction~\cite{polekhin2014examples} also can be considered here as an argument. 

We also believe that the presented topological approach, based on results from~
\cite{srzednicki2005fixed}, can be applied to different types of pendulum-like systems and might be considered as a possible alternative to functional analysis and variational approaches in applications of this type.

%% References


\bibliographystyle{elsarticle-num}
% \bibliographystyle{elsarticle-harv}
% \bibliographystyle{elsarticle-num-names}
% \bibliographystyle{model1a-num-names}
% \bibliographystyle{model1b-num-names}
% \bibliographystyle{model1c-num-names}
% \bibliographystyle{model2-names}
% \bibliographystyle{model3a-num-names}
% \bibliographystyle{model3-num-names}
% \bibliographystyle{model4-names}
% \bibliographystyle{model5-names}
% \bibliographystyle{model6-num-names}

%\bibliography{sample}

\bibliography{sample}

\end{document}

%% file: 2torus.pdf_tex
%% Creator: Inkscape 0.48+devel, www.inkscape.org
%% Accompanies image file '2torus.pdf' (pdf, eps, ps)
\begingroup%
  \makeatletter%
  \providecommand\color[2][]{%
    \errmessage{(Inkscape) Color is used for the text in Inkscape, but the package 'color.sty' is not loaded}%
    \renewcommand\color[2][]{}%
  }%
  \providecommand\transparent[1]{%
    \errmessage{(Inkscape) Transparency is used (non-zero) for the text in Inkscape, but the package 'transparent.sty' is not loaded}%
    \renewcommand\transparent[1]{}%
  }%
  \providecommand\rotatebox[2]{#2}%
  \ifx\svgwidth\undefined%
    \setlength{\unitlength}{589.99999025bp}%
    \ifx\svgscale\undefined%
      \relax%
    \else%
      \setlength{\unitlength}{\unitlength * \real{\svgscale}}%
    \fi%
  \else%
    \setlength{\unitlength}{\svgwidth}%
  \fi%
  \global\let\svgwidth\undefined%
  \global\let\svgscale\undefined%
  \makeatother%
  \begin{picture}(1,0.76949153)%
    \put(0,0){\includegraphics[width=\unitlength,page=1]{2torus.pdf}}%
    \put(0.03278444,0.63849878){\color[rgb]{0,0,0}\makebox(0,0)[lb]{\smash{$g$}}}%
    \put(0,0){\includegraphics[width=\unitlength,page=2]{2torus.pdf}}%
    \put(0.38067788,0.48353507){\color[rgb]{0,0,0}\makebox(0,0)[lb]{\smash{$F_{gravity}$}}}%
    \put(0.53951569,0.53583533){\color[rgb]{0,0,0}\makebox(0,0)[lb]{\smash{$F_{inertia}$}}}%
    \put(0.30861974,0.64740921){\color[rgb]{0,0,0}\makebox(0,0)[lb]{\smash{$F_{friction}$}}}%
    \put(0.4930266,0.62416463){\color[rgb]{0,0,0}\makebox(0,0)[lb]{\smash{$F_{constraint}$}}}%
  \end{picture}%
\endgroup%